\documentclass[12pt]{article}
\usepackage{geometry} % see geometry.pdf on how to lay out the page. There's lots.
\usepackage{ amssymb }
\usepackage{amsfonts} 
\usepackage{amsmath}                           
\usepackage{amsthm}
\usepackage{tikz-cd}

\usepackage{subfig}
\usepackage{pgfplots}
\usepackage{lipsum}

\usepackage[T1]{fontenc}
\usepackage[utf8]{inputenc}
\usepackage{authblk}
\title{Minimum depth of double cross product extensions}
\author[*]{Alberto Hern\'andez Alvarado}
\affil[*]{Escuela de Matem\'atica - Centro de Investiogaciones en Matem\'atica pura y Aplicada\\ Universidad de Costa Rica}
\date{} % delete this line to display the current date

\begin{document}

\maketitle

\newtheorem{thm}{Theorem}[section]
\newtheorem{ex}[thm]{Example}
\newtheorem{de}[thm]{Definition}
\newtheorem{lem}[thm]{Lemma}
\newtheorem{pr}[thm]{Proposition}
\newtheorem{co}[thm]{Corollary}
\newtheorem{rem}[thm]{Remark}

\newcommand \rrhu {\mathrel{\mathpalette\lrhup\relax}}
\newcommand \lrhup[2]{\ooalign{$#1\rightharpoonup$\cr $#2 \hspace{1mm}\rightharpoonup$\cr}}

\newcommand \llhu {\mathrel{\mathpalette\rlhup\relax}}
\newcommand \rlhup[2]{\ooalign{$#1\leftharpoonup$\cr $#2 \hspace{1mm}\leftharpoonup$\cr}}

\newcommand \vtick{\textsc{\char13}}
\newcommand{\exedout}{\rule{0.8\textwidth}{0.5\textwidth}}

%\tableofcontents

\begin{abstract}
In this paper we explore minimum odd and minimum even depth subalgebra pairs in the context of double cross products of finite dimensional Hopf algebras. We start by defining factorization algebras and outline how subring depth in this context relates with the module depth of the regular left module representation of the given subalgebra. Next we study minimum odd depth for double cross product Hopf subalgebras and determine their value in terms of their related module depth, we conclude that minimum odd depth of Drinfel'd double Hopf subalgebras is $3$. Finaly  we produce a necessary and sufficient condition for depth $2$ in double cross product Hopf subalgebra extensions. This sufficient condition is then used to prove results regarding minimum depth $2$ in Drinfel'd double Hopf subalgebras, particularly in the case of finite Group Hopf algebras. Lastly we provide formulas for the centralizer of a normal Hopf subalgebra in a double cross product scenario.
\end{abstract}



\textbf{Acknowledgements:} This research was funded by \textit{Escuela de Matem\'atica}  at \textit{Universidad de Costa Rica} via the project $821$-B$7$-$251$ -  \textit{CIMPA, UCR}. The author would also like to thank Yorck Sommerhauser for a fruitful conversation in Mexico City during the \textit{CLA} $2019$ regarding Section \eqref{Depth2}.

\section{Introduction and preliminaries}

The study of ring extensions and in particular finite dimensional algebra extensions has been central in the development of abstract algebra for the grater part of the last hundred years.  The concept of depth of a ring extension can be traced back to $1968$ to Hirata\vtick s work generalizing certain aspects of Morita theory \cite{Hi}. This work was followed by Sugano in \cite{Su}, and others throughout the nineteen seventies and the nineteen eighties  such as \cite{Su1}, \cite{Mor} and \cite{Ste}

 In $1972$ W. Singer introduced the idea of a matched pair of Hopf algebras in the connected case \cite{Si}, this was extended by Takeuchi \cite{Ta} in the early nineteen eighties by considering the non connected case. Both these works set the basis for the study of double cross products of Hopf algebras that was brought forward by Majid \cite{Ma1} and others, starting from the early nineteen nineties.
 
More recently in  the early two thousands the idea of depth of a ring extension was further studied in the context of Galois coring structures \cite{Ka1}, and to characterize structure properties involving self duality, Forbenius extensions  and normality such as in \cite{Ka7}, \cite{Bu1} and \cite{KK1}. Moreover,  fair amount of research regarding combinatorial aspects of finite group extensions has been done recently as well, we point out  \cite{BK1} and \cite{BDK}. Other interesting results may be found in \cite{BKK}, \cite{SD}, \cite{F} and \cite{FKR}.  The other variant of this trend that has been developing in recent years is the study of depth in the context of finite dimensional Hopf algebra extensions  \cite{BK1}, \cite{H}, \cite{HKY}, \cite{HKS}, \cite{HKL}, \cite{Ka2}, \cite{Ka8} and others. Is in the spirit of the latter that we develop the work presented here, in the context of extensions of finite dimensional Hopf algebras in double cross products of Hopf algebras.

Throughout this paper all rings $R$ and algebras $A$ are associative with unit, all algebras are  finite dimensional over a field  $k$ of characteristic zero. All modules  $M$ are finite dimensional as well.  All subring pairs $S \subseteq R$ satisfy $1_S = 1_R$ and we denote the extension as $S\hookrightarrow R$.

The paper is organized as follows: In Subsection \eqref{preliminaries} preliminaries on the concept of depth will be reviewed. Mainly definitions on subring depth, the concept of module depth in a tensor category and some results that will be of interest further into this study. Other concepts will be introduced when needed.
 
 Section \eqref{Factorisation Algebras} deals with the concept of an algebra extension that factorizes as a tensor product of subalgebras. We adapt the concept of subring depth to this scenario and prove two preliminary results on depth of an extension of an algebra in a factorization algebra in Theorems \eqref{factorization iso}, \eqref{thm45} and Corollary \eqref{augmented algebra equality}. Example \eqref{Heisenberg} reviews the case of the minimum depth of a Hopf algebra $H$ in its smash product with and $H$-module algebra $A$, in particular the case of the Heisenberg double $\mathcal{H}(H)$ of a finite dimensional Hopf algebra, which motivates the next two sections.
 
 Section \eqref{Double Cross Products} deals with the definitions of double cross products as factorization algebras  in Propositions \eqref{dcp}  and \eqref{FactorHalgebra} and explores minimum odd depth for this cases in Theorems   \eqref{mddcp} and \eqref{Drinfelddepth}.
 
 Section \eqref{Depth2} contains our main result in the form of Theorem \eqref{mainthm}. Our result establishes  a necessary and sufficient condition for minimum even depth to be less or equal to $2$ in the case of double cross product  extensions of Hopf algebras. This sufficient condition is then utilized to prove particular cases for Drinfel\vtick d double extensions in the case of finite group algebras in  Corollary \eqref{abeliangroup} and to provide formulas for the centralizer of a Hopf subalgebra in the case of a depth two double cross product extension in Proposition \eqref{centralizer} and Corollary \eqref{centralizergroup}.

\subsection{Preliminaries on Depth}
\label{preliminaries}

Let $R$ be a ring and $M$ and $N$ two left (or right) $R$-modules. We say $M$ is similar to $N$ as an $R$ module if there are positive integers $p$ and $q$ such that $M|pN$ and $N|qM$, where $nV$ means $\oplus^n V$ for every $n$ and every $R$ module $V$ and $M|pN$ means that  $M$ is a direct summand of $pN$ or equivalently that $M\otimes * \cong pN$, in this case we denote the similarity as $M \sim N$. Notice that this similarity is compatible with induction and restriction functors on $_R\mathcal{M}$, for if $R\hookrightarrow L$ is an extension of $R$ and $K$ is an  right $L$ module then $M\sim N$  as $R$ modules implies $M\otimes_R K \sim N\otimes_R K$ as right $L$ modules. Moreover, if $S\hookrightarrow R$ is a subring then $M\sim N$ as $R$ modules implies $M\sim N$ as $S$ modules.

Consider now a ring extension $B\hookrightarrow A$. Let $n \geq 1$, by  $A^{\otimes_B (n)}$ we mean $A\otimes_B A\otimes_B \cdots \otimes_B A$ $n$ times, and define $A^{\otimes_B (0)}$ to be $B$. Notice that for $n\geq 1$ $A^{\otimes_B (n)} $ has a natural $X$-$Y$-bimodule structure where $X,Y \in \{A,B\}$ and for $n=0$ we get a $B$-$B$-bimodule structure.

\begin{de}
\label{depthdefinition}

Let $B\hookrightarrow A$ be a ring extension, we say $B$ has:
\begin{enumerate}
\item \textbf{Minimum odd depth 2n +1}, denoted $d(B,A) = 2n+1$, if $A^{\otimes_B (n+1)} \sim A^{\otimes_B (n)}$ as $B$-$B$ modules for  $n\geq 0$.
\item \textbf{Minimum even depth 2n}, denoted d(B,A) = 2n,   if $A^{\otimes_B (n+1)} \sim A^{\otimes_B (n)}$ as either $B$-$A$ or $A$-$B$ modules for for $n\geq1$.
\end{enumerate}
\end{de}

Notice that by the observation made above one has that for all $n\geq 0$ $d(B,A) = 2n$ implies $d(B,A) = 2n+1$ by module restriction, and that  for all $m\geq 1$ $d(A,B) = 2m +1$ implies $d(B,A) = 2m + 2$ for all $m$ by module induction. Hence we are only interested in the minimum values for which any of these relations is satisfied. In case there is no such minimum value we say the extension has infinite depth.

A third type of subring depth called \textbf{H-depth} denoted by  $d_h(B,A) = 2n-1$ if $A^{\otimes_B (n+1)} \sim A^{\otimes_B (n)}$ as $A$-$A$ modules for $n\geq 1$ was introduced by Kadison in \cite{Ka5} as a continuation of the study of $H$ - separable extensions introduced by Hirata, where such extensions are exactly the ones satisfying $d_h(B,A) = 1$. For the purposes of this paper we will restrict our study to minimum odd and even depth only. In particular the cases $d(B,A) \leq 3$ and $d(A,B) \leq 2$.

Let $B\hookrightarrow A$ be a ring extension, $R = A^B$ the centralizer and $T = (A\otimes_B A)^B$ the $B$ central tensor square. It is shown in \cite{Ka1}[Section 5] that $d(A,B) \leq 2$ implies a Galois $A$-coring structure in $A\otimes_R T$ in the sense of \cite{BzWs}. Further more it is also shown in \cite{Ka1} that if the extension $B\hookrightarrow A$ is Hopf Galois for a given finite dimensional Hopf algebra $H$ then $d(B,A) \leq 2$.

Let $R\hookrightarrow H$ be a finite dimensional Hopf algebra extension. Define their quotient module $Q$ as $H/R^+H$ where $R^+ = ker\varepsilon \cap R$.  Suppose that $R$ is a normal Hopf subalgebra of $H$, one can easily show that  the extension $R\hookrightarrow H$ is $Q$-Galois and therefore $d(R,H) \leq 2$. The converse happens to be true as well and the details can be found in \cite{BK1}[Theorem 2.10]. Hence,  the following result  holds:

\begin{thm}
\label{normality}
Let $R\hookrightarrow H$ be a finite dimensional Hopf algebra pair. Then $R$ is a normal Hopf subalgebra of $H$ if and only if 
\[d(R,H) \leq 2\]
\end{thm}

Now we consider again a $k$ algebra $A$ and an $A$-module $M$. Recall that the $n$-th truncated tensor algebra of $M$ in $_A\mathcal M$ is defined as 
\[T_n(M) = \bigoplus_{i=1}^n M^{\otimes (n)}\quad \textit{and} \quad T_0(M) = k\]

We then define the \textbf{module depth} of $M$ in $_A\mathcal{M}$ as $d(M,_A\mathcal{M}) = n$ if and only if $T_n(M) \sim T_{n+1}(M)$. In case $M$ is an $A$-module coalgebra (a coalgebra in the category of $A$ modules) then $d(M,_A\mathcal{M}) = n$ if and only if $M^{\otimes (n)} \sim M^{\otimes (n+1)}$ \cite{Ka2}, \cite{HKY}.

We point out that an $A$-module $M$  has module depth $n$ if and only if it satisfies a polynomial equation $p(M) = q(M)$ in the representation ring of $A$. In this case $p$ and $q$ are polynomials of degree at most $n+1$ with integer coefficients. A brief proof of this can be found in \cite{H}. For this reason we say that a module $M$ has  finite module depth in $_A\mathcal{M}$ if and only if it is an algebraic element in the representation ring of $A$.

Finally, we would like to mention that in the case of Hopf subalgera extensions $R\hookrightarrow H$ there is a way to link subalgebra depth with module depth. The reader will find a proof of the following in \cite{Ka2}[Example 5.2]:

\begin{thm}
\label{hopfmoduledepth}
Let $R\hookrightarrow H$ a Hopf subalgebra pair. Consider their quotient module $Q$, then the minimum depth of the extension satisfies:
\[2d(Q,_R\mathcal{M}) + 1 \leq d(R,H) \leq 2d(Q,_R\mathcal{M}) + 2\]
\end{thm}

\section{Depth of factorization algebra extensions}
%\subsection{Factorisation algebras and depth}
\label{Factorisation Algebras}

Let $A$ and $B$ be two finite dimensional algebras. Consider the following map:
\[
\psi : B\otimes A \longrightarrow A\otimes B \quad ; \quad b\otimes a \longmapsto a_{\alpha}\otimes b^{\alpha}  
\]

such that

\[
\psi (1_B\otimes a) = a\otimes 1_B,\quad \psi (b\otimes 1_A) = 1_A\otimes b
\]

for all $a \in A$ and $b\in B$. Moreover suppose $\psi$ satisfies the following commutative octagon for all $a,d \in A$, and all $b,c \in B$:

\begin {equation}
\label{factorisation associativity}
(ad_\alpha)_\beta \otimes b^\beta c^\alpha = a_\beta d_\alpha \otimes (b^\beta c)^\alpha
\end{equation}

We call $\psi$ a factorisation of $A$ and $B$ and $A\otimes_\psi B$ a factorisation algebra of $A$ and $B$, a  unital associative algebra with product

\begin{equation}
(a\otimes b) (c\otimes d) = a\psi (b\otimes c)d = ac_\alpha \otimes b^\alpha d
\end{equation}\\

where $a,c \in A$, $b,d \in B$ and the unit element is $1_A \otimes 1_B$.  Besides $A$ and $B$ are $A\otimes_\psi B$ subalgebras via the inclusions $A\hookrightarrow A\otimes_\psi 1_B$ and $B\hookrightarrow 1_A\otimes_\psi B$.

Factorisation algebras are ubiquitous: Setting $\psi(b\otimes a) = a\otimes b$ yields the tensor algebra $A\otimes B$. If $H$ is a Hopf algebra and $A$ a left $H$-module algebra satisfying $h\cdot (ab) = (h_{1}\cdot a )(h_2 \cdot b)$, $h\cdot 1_A = \varepsilon(h) 1_A$ for all $h \in H$ and $a,b \in A$, define $\psi: H\otimes A; \quad h\otimes a\longmapsto h_1\cdot a \otimes h_2$ then the product becomes $(a\otimes h)(b\otimes g) = a\psi(h\otimes b)g = a(h_1 \cdot b \otimes h_2)g = ah_1\cdot b\otimes h_2 g$. It is a routine exercise to verify that $A\otimes_\psi H$ is a factorisation algebra and that $ A\otimes_\psi H = A\# H$ is the smash product of $A$ and $H$. Double cross products of Hopf algebras are also examples of factorisation algebras, we will study them further in Section \eqref{Double Cross Products}.

Now let $A\otimes_\psi B$ be a factorisation algebra via $\psi: B\otimes A \longmapsto A\otimes B$. For the sake of brevity we will denote it $S_\psi = A\otimes_\psi B$ for the rest of this Section. We point out that due to multiplication in $S_\psi$ and the fact that both $A$ and $B$ are subalgebras of $S_\psi$ we get that for every $n\geq 1$, $S_\psi^{\otimes_B (n)} \in$ $_{S_\psi}\mathcal{M}_{S_\psi}$ in the following way:
\begin{equation*}
(a\otimes_\psi b)(a_1 \otimes b_1 \otimes_B \cdots \otimes_B a_n\otimes b_n)(c \otimes_\psi d) =
\end{equation*}  
\begin{equation*}
= a\psi(b\otimes a_{1})b_{1}\otimes_{B}\cdots \otimes_{B}a_{n}\psi(b_{n}\otimes c)d =
\end{equation*}
\begin{equation}
\label{psi linearity}
= aa_{1\alpha} \otimes b^\alpha b_1\otimes_B \cdots \otimes_B a_nc_\alpha\otimes b_n^\alpha d
\end{equation}

The same condition holds for $S_\psi$ as either left or right $B$ module via subalgebra restriction. In this case we can assume $n \geq 0$ and define $S_\psi^{\otimes_B (0)} = B$. This allows us to consider the following isomorphism:

\begin{thm}
\label{factorization iso}
Let $A$ and $B$ be algebras, $\psi: B\otimes A \longmapsto A\otimes B$ a factorisation and $S_\psi$ the corresponding factorisation algebra. Then:
\begin{equation}
S_\psi ^{\otimes_B (n)} \cong A^{\otimes(n)} \otimes B
\end{equation}
as $X$-$Y$-bimodules, with $X$,$Y$ $\in \{S_{\psi}, B \}$ for $n\geq 1$ and as $B$-$B$-bimodules for $n \geq 0$.
\end{thm}

\begin{proof}

First notice that for $n = 1$, $A\otimes_{\psi} B \cong A\otimes B$ via $a\otimes_{\psi}b \longmapsto a\otimes b$, since $A\otimes_{\psi} B$ is an algebra and multiplication is well defined.

Now, for every $n > 1$, $(A\otimes_{\psi} B)^{\otimes_{B} (n)}$ $\cong (A\otimes_{\psi} B)^{\otimes_{B} (n-1)} \otimes_{B} (A\otimes_{\psi} B)$. By induction on $n$ and using that $B\otimes_{B} A \cong A$ one gets:

\begin{equation*}
(A\otimes_{\psi} B)^{\otimes_{B}{(n-1)}} \otimes_{B} A\otimes_{\psi} B \cong  A^{\otimes_{B}(n-1)} \otimes B \otimes_{B} A\otimes B 
\end{equation*} 
\begin{equation}
\cong A^{\otimes (n-1)} \otimes A\otimes B  \cong A^{\otimes (n)} \otimes B
\end{equation} 

Finally for $n = 0$ we get $S_\psi^{\otimes_B (0)} = B \cong k\otimes B \cong A^{\otimes (0)} \otimes B$ as $B$-$B$ bimodules.

\end{proof}

Recall that a Krull-Schmidt category is a generalization of categories where the Krull-Schmidt Theorem holds. They  are additive categories such that each object decomposes into a finite direct sum of indecomposable objects having local endomorphism rings, also this decompositions are unique in a categorical sense. For example categories of modules having finite composition length are Krull-Schmidt.  

Theorem \eqref{factorization iso} in the context of a Krull-Schmidt category, allows us relate subalgebra depth in a factorization algebra with module depth in the finite tensor category of finite dimensional left $B$-modules. In turn this will allow us to compute minimum odd depth values in the case of Smash Product algebras and Drinfel \vtick d Double Hopf algebras at the end of this Section as well as in Section \eqref{Double Cross Products}. The next Theorem and its Corollary  provide this connection and they mirror  \cite{Ka2}[Equation 21] and \cite{HKY}[Equation 21].

\begin{thm}
\label{thm45}
Let $A\otimes_\psi B$ be a factorisation algebra with $_{B}\mathcal{M}_B$ a Krull-Schmidt category, and $A \in$$_B\mathcal{M}$ Then the minimum odd depth of the extension satisfies:
\begin{equation}
\label{factorisation h depth}
d(B,S_\psi) \leq 2d(A, _B\mathcal{M}) + 1
\end{equation}
\end{thm}

\begin{proof}
Let $d(A,_B\mathcal{M}_B) = n$. Since $_B\mathcal{M}_B$ is a Krull-Schmidt category, standard face and degeneracy functors imply $A^{\otimes_B (m)} | A^{\otimes_B (m+1)}$ for $m\geq 0$. Then $T_n(A) \sim T_{n+1}(A)$ implies $A^{\otimes (n+1)} \sim A^{\otimes (n)}$. Tensoring on the right by $(- \otimes B)$ one gets $A^{\otimes (n+1)}\otimes B \sim A^{\otimes (n)}\otimes B$. By Theorem \eqref{factorization iso} this is equivalent to $(A\otimes_\psi B)^{\otimes_B (n+1)} \sim (A\otimes_\psi B)^{\otimes_B (n)}$. This by definition is $d(B,S_\psi) \leq 2n + 1$.
\end{proof}

Recall that $B$ is a bialgebra if it is both an algebra and a coalgebra such that the coalgebra morphisms are algebra maps, i.e. $B$ is a coalgebra in the category of $k$ algebras. This means that the counit  $\varepsilon : B \longrightarrow k$ is an algebra map that splits the coproduct: $(\varepsilon \otimes id)\circ \Delta = (id \otimes \varepsilon)\circ \Delta = id$.  Via the counit the ground field $k$ becomes a trivial $B$ module via $b\cdot k = \varepsilon(b)k$. Hence, a $k$ vector space $V$ becomes a right $B$-module via : $V \cong V\otimes k$.
\begin{co}
\label{augmented algebra equality}
Let $B$ be a bialgebra. Then the inequality \eqref{factorisation h depth} becomes an equality. 
\end{co}

\begin{proof}
Let  $B$ be a  bialgebra, since $k$ becomes a $B$-module via the counit of $B$, tensoring by $-\otimes_B k$ or $k\otimes_B -$ is a morphism of $B$ modules. Let $d(B, S_\psi) = 2n + 1$, then by definition $S_\psi^{\otimes_B (n)} \sim S_\psi^{\otimes_B (n+1)}$ as $B$-$B$ bimodules, and by the isomorphism in Theorem \eqref{factorization iso} this implies $A^{\otimes (n)} \otimes B \sim A^{\otimes (n+1)} \otimes B$, then it suffices to tensor on the right by $ (- \otimes_B k) $ on both sides of the similarity  to get $A^{\otimes n+1} \sim A^{\otimes n}$ which in turn implies $d(A,_B\mathcal{M}) \leq n$.
\end{proof}

Notice that assuming that $A \in _B\mathcal{M}$ makes sense since the factorization algebras we are considering next all depend on this fact to be well defined. On the other hand this result says nothing about $even$ depth since by no means one should expect $A$ to be a right  or left $S_\psi$-module.

\begin{ex}
\label{Heisenberg}
\cite[Theorem 6.2] {HKY} Let $H$ be a Hopf algebra and $A$ an $H$-module algebra, consider their smash product algebra $A\#H$ and the algebra extension $H \hookrightarrow A\#H$. The extension satisfies:
\[
d(H,A\#H) = d(A,_H\mathcal{M}) + 1
\]
Moreover, as a consequence of this one can show the following: Let $dim_k(H)\geq 2$ and consider $H^*$ as a $H$-module algebra via $h\rightharpoonup f$ and their smash product  $H^*\#H$, also known as their Heisenberg double, then the extension $H\hookrightarrow H^*\#H$ satisfies 
\[ d(H,H^*\#H) = 3
\]
This follows since $H$ is a factor $H^*\# H$ subalgebra and the fact that $_{H^*}H \cong H^*_{H^*}$ and that minimum depth satisfies $d(H^*,\mathcal{M}_{H^*}) = 1$.
\end{ex}

This example motivates the question of whether this result (or an equivalent one) can be attained for a more general class of extensions of Hopf algebras into factorization algebras. The next two Sections deal with this question in the context of the Drinfel\vtick d double $D(H)$ of a Hopf algebra and more generally in the case of the double cross product $A\bowtie B$ of a matched pair of Hopf algebras $A$ and $B$.

\section{Double cross products and minimum odd depth}
\label{Double Cross Products}

The study of double cross products was started in the early seventies by  W. Singer with the introduction of matched pairs of Hopf algebras satisfying certain module-comodule factorization conditions in the case of connected module categories, \cite{Si}. Later M. Takeuchi \cite{Ta} furthered the study of matched pairs in the ungraded case, in particular, he aimed at describing natural properties of braided groups. Later S. Majid \cite{Ma1} studied bicrossed products as a means to construct self dual objects in the category of Hopf algebras primarily in the case of non commutative non cocommutative cases, in some sense motivated by the possibility to construct models for  quantum gravity. We follow Majid\vtick s definition of double cross products as in \cite{Ma}.

Let $A$ and $B$ be two Hopf algebras such that $A$ is a right $B$-module coalgebra and $B$ a left $A$-module coalgebra. We say $B$ and $A$ are a matched pair \cite{Ma}[Definition 7.2.1] if there are coalgebra maps 
\[
\alpha: A\otimes B \longrightarrow A; \quad h\otimes k \longmapsto h\triangleleft k \quad \textit{and}\quad  \beta : A\otimes B \longrightarrow B; \quad h\otimes k \longmapsto h\triangleright k
\]

 such that the following compatibility conditions hold:
 
 \begin{equation}
 (hg)\triangleleft k = \sum (h\triangleleft(g_1\triangleright k_1))(g_2\triangleleft k_2) ; \quad 1_A\triangleleft k = \varepsilon_B(k) 1_A
 \end{equation}
 and 
 \begin{equation}
 h\triangleright (kl) = \sum (h_1\triangleright k_1)((h_2\triangleleft k_2)\triangleright l) ; \quad h \triangleright 1_B = \varepsilon_A(h) 1_K
 \end{equation}
 
 Define a product by 
 \begin{equation}
 (k\bowtie h)(l\bowtie g) = \sum k(h_1\triangleright l_1)\bowtie (h_2\triangleleft l_2)g
 \end{equation}
 the resulting algebra $B\bowtie A$ is called the $\mathbf{double}$ $\mathbf{crossed}$ $\mathbf{product}$ of $A$ and $B$ \cite[Theorem 7.2.2]{Ma}, and is a Hopf algebra with coproduct, counit and antipode given by 
 
 \begin{equation}
 \Delta(k\bowtie h) = k_1\bowtie h_1 \otimes k_2\bowtie h_2
 \end{equation}
 \begin{equation}
 \varepsilon (k\otimes h) = \varepsilon_K(k)\varepsilon_H(h)
 \end{equation}
 \begin{equation}
  S(k\bowtie h) = (1_K\bowtie S_H(h))(S_K(k)\bowtie k) 
 \end{equation}
 \begin{equation*}
 = S_H(h_1)\triangleright S_K(k_1) \bowtie S_H(h_2)\triangleleft S_K(k_2) 
 \end{equation*}
 respectively.

The following are well known results and are cited here for the sake of completeness, they summarize the fact that Double Cross Products of Hopf algebras are exactly the Hopf algebras that factorize as the product of two Hopf subalgebras. The reader can refer to them in \cite{Ma} and \cite{Ma1} as well as in  \cite{BCT}.

\begin{pr}
\label{dcp}
Double crossed products are factorisation algebras
\end{pr}

The converse is also true:

\begin{pr}\cite[Theorem 2.7.3]{Ma}
\label{FactorHalgebra}
Suppose $H$ is a Hopf algebra and $L$ and $A$ two sub-Hopf algebras, such that $H \cong A\otimes_\psi L$ is a factorisation, then  $H$ is a double crossed product.
\end{pr}

\begin{proof}
The multiplication $m: L \otimes A \longrightarrow H$ defined by $a\otimes l \longmapsto al$ is a bijection. This implies $\displaystyle{A\bigcap L = k}$. Then consider the map:
\begin{equation*}
\mu: L\otimes A \longrightarrow A\otimes L; \quad l\otimes a \longmapsto m^{-1}(la)
\end{equation*}
then define 
\begin{equation*}
\triangleright: L\otimes A\longrightarrow A; \quad l\triangleright a = ((\varepsilon_L\otimes Id)\circ \mu)(l\otimes a)
\end{equation*}
\begin{equation*}
\triangleleft: L\otimes A\longrightarrow L; \quad l\triangleleft a = ((Id\otimes \varepsilon_A)\circ \mu)(l\otimes a)
\end{equation*}
\end{proof}

We wrote the proof of this last Proposition since it allows us to construct examples such as Example \eqref{tensorexample}. \\

Now, let $H$ be any Hopf algebra with bijective antipode $S$ with composition inverse $\overline{S}$. Let $S^*$ be the bijective antipode of $H^*$ and $\overline{S^*}$ its composition inverse, then $H$ is a right $H^{*cop}$-module coalgebra via  
\[
h \llhu f = \sum \overline{S^*}(f_2) \rightharpoonup h\leftharpoonup f_1
\]
 and $H^*$ is a left  $H$-module coalgebra via 
 \[
 h \rrhu f = \sum h_1\rightharpoonup f \leftharpoonup \overline{S}(h_2)
 \]
 see \cite{Mo}[Chapter 10] for details on this actions. Define the Drinfel\vtick d double of $H$, $D(H)$ as the double cross product $H^{*cop}\bowtie H$ with product 
 \[
 (f\bowtie h)(g\bowtie k) = \sum f(h_1\rrhu g_2)\bowtie(h_2\llhu g_1)k
 \]
  The coproduct, counit and antipode are given by 
  \[
  \Delta(f\bowtie h) = \sum (f_2\bowtie h_1)\otimes (f_1\bowtie h_2)
  \]
  \[
  \varepsilon_{D(H)}(f\bowtie h) = \varepsilon_{H^*}(f)\varepsilon_H (h)
  \]
   and 
   \[
   S_{D(H)} (f\bowtie h) = \sum (S(h_2) \rightharpoonup S(f_1)) \bowtie (f_2 \leftharpoonup S(h_1))
   \]
   respectively.

Since double crossed products of Hopf algebras are both factorization algebras and Hopf algebras  Corollary \eqref{augmented algebra equality} becomes:

\begin{pr} 
\label{mddcp}
Let $H$ and $K$ be a matched pair of Hopf algebras and consider their double crossed product $H\bowtie K$, then the Hopf algebra extension $H \hookrightarrow H\bowtie K$ satisfies 
\[
d(H,H\bowtie K) = 2d(K, _H\mathcal{M}) + 1
\] 
\end{pr}

\begin{ex}
\label{tensorexample}
Recall that two Hopf algebras $A$ and $B$ are said to be paired \cite{Ma1}[1.4.3] if there is a bilinear map 
\[A\otimes B \longrightarrow k; a\otimes b\longmapsto \langle a,b\rangle\]
Satisfying $\langle ac,b\rangle = \langle a\otimes c,\Delta b\rangle$, $\langle a,1\rangle = \varepsilon(a)$, $\langle 1,b\rangle = \varepsilon(b)$ and $\langle Sa,b \rangle = \langle a,Sb\rangle$. We say also say it is nondegenerate if and only if $\langle a,b \rangle = 0 $ for all $b\in B$ implies $a=0$ and $\langle a,b\rangle = 0$ for all $a \in A$ implies $b =0$. Assume now that $A$ and $B$ are paired and that $\langle , \rangle$ is convolution invertible, define 
\[a\triangleleft b = \sum a_2\langle a_1,b_1\rangle^{-1}\langle_3,b_2\rangle \]
\[a\triangleright b = \sum b_2\langle a_,b_1\rangle^{-1}\langle a_2,b_3\rangle\]
With this action we can endow $A^{op}\bowtie B$ with a double cross product structure.
Consider then $H$ to be a finite dimensional Hopf algebra and 
\[\langle , \rangle  : H\otimes H \longrightarrow k; h\otimes g \longmapsto \varepsilon(a)\varepsilon(b)\]
then $\langle , \rangle $ satisfies the conditions above, is nodegenerate if and only if $H$ is semisimple via Maschke \vtick s theorem and is convolution invertible via $\langle , \rangle  \langle , \rangle = \varepsilon $ Then $H^{op} \bowtie H$ is a double cross product  isomorphic to the tensor Hopf algebra $H^{op} \otimes H$, Proposition \eqref{FactorHalgebra},  and the minimum odd depth satisfies
\[d(H,H^{op}\bowtie H) = 3\]
Since  $d(H,_{H^{op}}\mathcal{M}) = 1$.
\end{ex}

\begin{pr}
\label{Drinfelddepth}
Let $H$ be a finite dimensional Hopf algebra  of dimension $m \geq 2$ and consider $D(H) = H^{*cop} \bowtie H$ its Drinfel\vtick d double. Then  the minimum odd depth satisfies:
\[ d(H,D(H)) = 3 \]
\end{pr}

\begin{proof}
The proof is analogous to Example \eqref{Heisenberg}. 
\end{proof}

This result should not come as a surprise: Whenever $H$ is cocommutative it is easy to show that its Drinfel\vtick d double and its Heisenberg double are isomorphic as algebras and since given two isomorphic algebras $A$ and $B$ and an $A$-module $M$, module depth satisfies $d(M,_A\mathcal{M}) = d(M,_B\mathcal{M})$, is immediate that for cocommutative $H$, minimum odd depth is given by  $d(H,D(H)) = d(H, H^*\#H) = 3$. But it is straightforward that depth does not depend on the cocommutativity of the coalgebra structure on $H$.

%Moreover, a consequence of this is a well known result  of the theory of the representations of finite groups:

%\begin{co}
%Let $G$ be a finite centerless group, $k$ an algebraically closed field of characteristic zero. Then $(kG)_{ad}$ the regular %adjoint representation is faithful as a $G$ module.
%\end{co}  

%\begin{proof}
%Consider the Hopf algebra extension $(kG) \hookrightarrow D(kG)$ (see details bellow in  Section\eqref{Depth2}), following %\cite{HKY}[Section 5] one proves that the generalized quotient module $Q = D(kG)\backslash (kG)^+ D(kG)$ is isomorphic to $%(kG)^*_{ad}$ via $\overline{p_y\bowtie h} \longmapsto p_{h^{-1}yh}$ and $(kG)_{ad} \cong (kG)^*_{ad}$, hence $Q \cong %(kG)_{ad}$. By Proposition  \eqref{Drinfelddepth} $d(kG,D(kG)) = 3$, apply then  Theorem \eqref{hopfmoduledepth} to get %$d(Q,_{kG}\mathcal{M}) = 1$ and  by \cite{HKY}[Proposition 5.2]  conclude $Q \cong (kG)_{ad}$ is faithful.
%\end{proof}

\section{Depth two}
\label{Depth2}

Consider a finite group algebra $kG$ and its dual  $(kG)^* = k\langle p_x | x \in G\rangle$  where the $\{p_x\}$ form the dual basis of $G$ satisfying $p_x(y) = \delta_{x,y}$ for all $x,y \in G$. This is an algebra via convolution product and the identity element is $\varepsilon = \sum_{y\in G} p_y$. $(kG)^*$ has a Hopf algebra structure given by 
\[ \Delta^* p_x = \sum_{lk=x} p_l \otimes p_k \]
\[\varepsilon^*(p_x) = \delta_{x,1}\]

And antipode $S^*$.

Consider then $R = kG$ a finite group algebra and $H = D(kG) = (kG)^{*cop} \bowtie kG$ its Drinfel\vtick d double. Multiplication is given by
\[(p_x\bowtie g)(p_y\bowtie k) = p_xp_{gyg^{-1}} \bowtie gk\] 
and the antipode is 
\[S(p_x\bowtie g) = (\varepsilon \bowtie g^{-1})(S^*p_x \bowtie e) = S^*p_{g^{-1}xg}\bowtie g^{-1}\]

Let now $p_x = p_x \bowtie e \in (kG)^*$, and $p_y\bowtie g \in H$. The right adjoint action of $H$ on $(kG)^*$ is given by 
\[S(p_y \bowtie g)_1(p_x\bowtie e)(p_y\bowtie g)_2 = \sum_{lk = y}S(p_l \bowtie g)(p_x\bowtie e)(p_k\bowtie g)\]
a quick calculation and using the formulas above shows that the latter equals 
\[\sum_{lk=y}S^*(p_{g^{-1}lg})p_{g_{-1}xg}p_{g^{-1}kg} \bowtie e\] 
A similar calculation shows that the left adjoint action of $H$ on $(kG)^*$ yields 
\[(p_y\bowtie g)_1 (p_x \bowtie e)S(p_y \bowtie g)_2 =   \sum_{lk=y} p_lp_{gxg^{-1}}p_k \bowtie e\]
and hence $(kg)^*$ is $H$ left and right ad stable and hence normal. As it is shown in Theorem \eqref{normality} this implies then that 
\[d((kG)^*, H) \leq 2\]

We point out that this is true since the left coadjoint action of $(kG)^*$ on $kG$  given by $\llhu$ is trivial on the generators:
\[g\llhu p_x = g\]
The following theorem tells us that this is in fact a necessary and sufficient condition for depth $2$ in the more general case of double cross products:
 
\begin{thm}
\label{mainthm}
Let $A,B$ be a matched pair of Hopf algebras and let $H = A\bowtie B$ be their double cross product. Then $d(A,H) \leq 2$ (Equivalently $d(B,H)\leq 2$)if and only if $B\triangleleft A$ (Equivalently $B\triangleright A$) is trivial.
\end{thm}

\begin{proof}
Let $A\bowtie B$ be a double cross product of Hopf algebras. Recall that an extension of finite dimensional Hopf algebras has depth $ \leq 2$ if and only if the extension is normal. Let $a\bowtie 1_B \in A$ and $h\bowtie g \in A\bowtie B$. Consider the right adjoint action of $A\bowtie B$ on $A$:
\[S(h_1\bowtie g_1)(a\bowtie 1_B)(h_2\bowtie g_2) = (Sg_1\triangleright Sh_1\bowtie Sg_1 \triangleleft Sh_1)(a\bowtie 1_B)(h_2\bowtie g_2)\]
\[ = ((Sg_1\triangleright Sh_1)((Sg_2\triangleleft Sh_2)\triangleright a_1))(((Sg_3\triangleleft Sh_3)\triangleleft a_2)\triangleright h_4) \]
\[\bowtie (((Sg_4\triangleleft Sh_5)\triangleleft a_3)\triangleleft h_6)g_5\]
$\in A$ if and only if 
\[ (((Sg_4\triangleleft Sh_5)\triangleleft a_3)\triangleleft h_6)g_5 = \lambda 1_B \]
for some $\lambda \in k$. 

Suppose  that $B\triangleleft A$ is trivial, then 
\[(((Sg_4\triangleleft Sh_5)\triangleleft a_3)\triangleleft h_6)g_5 = Sg_4\varepsilon(Sh_5)\varepsilon(a_3)\varepsilon(h_5) g_5\]
\[ = \varepsilon(g_4)\varepsilon(h_5)\varepsilon(a_3)1_B\]
Take $\lambda = \varepsilon(g_4)\varepsilon(h_5)\varepsilon(a_3)$.

Now assume that 
\[(((Sg_4\triangleleft Sh_5)\triangleleft a_3)\triangleleft h_6)g_5 = \lambda 1_B \]
for some $\lambda \in k$. Without loss of generality we can assume $h = 1_A$ so we obtain 
\[(((Sg_4\triangleleft Sh_5)\triangleleft a_3)\triangleleft h_6)g_5 = (Sg_3\triangleleft a_3)g_4 = \lambda 1_B\]
apply $\varepsilon$ on both sides of the equation to obtain 
\[\varepsilon(g_3)\varepsilon(a_3) = \lambda\]
Now let $g \in B$ and $a\in A$ since the antipode is bijective let $g = Sh$, then 
\[g\triangleleft a = Sh\triangleleft a = (Sh_1\triangleleft a)h_2Sh_3 = \varepsilon(Sh_1)\varepsilon(a)Sh_2 = Sh\varepsilon(a) = g\varepsilon(a)\]

Then $A$ is $A\bowtie B$ right ad-stable if and only if $B\triangleleft A$ is trivial. 

Consider now the left adjoint action of $A\bowtie B$ on $A$.  Then 
\[(h_1\bowtie g_1)(a\bowtie 1_b)S(h_2\bowtie g_2) \in A\]
if and only if 
\[[(g_3\triangleleft a_3)\triangleleft(Sg_5\triangleright Sh_3)](Sg_6\triangleleft Sh_4) = \lambda 1_B\]
for some $\lambda \in k$. The rest of the proof mirrors what was done above and then $A$ is left $A\bowtie B$ ad stable if and only if $B\triangleleft A$ is trivial, hence the extension is normal if and only if $B\triangleleft A$ is trivial  and $d(A,A\bowtie B) \leq 2 $ if and only if $B\triangleleft A$ is trivial. The case of the extension $B\hookrightarrow A\bowtie B$ is symmetric.
\end{proof}

\begin{co} 
\label{abeliangroup}
Let $G$ be a finite group and consider $D(kG)$, then 
\[d(kG, D(kG)) \leq 2\]
if and only if $G$ is abelian.
\end{co}

\begin{proof}
Let $g,x \in G$. Recall that the left coadjoint action of $kG$ on $(kG)^*$ is given by $g\rrhu p_x = p_{gxg^{-1}}$ which is trivial (i.e $ p_{gxg^{-1}} = p_x$ for all $g,x \in G$) if and only if $G$ is abelian. 
\end{proof}

\begin{ex}
Consider $H^{op}\bowtie H$ as in Example \eqref{tensorexample}, then the minimum depth satisfies 
\[d(H,H^{op}\bowtie H) \leq 2\]
since $h \triangleright g = \sum g_2\langle h_1,g_1\rangle^{-1}\langle h_2,g_3\rangle  = g_2\varepsilon(h_1)\varepsilon(g_1)\varepsilon(h_2)\varepsilon(g_3) = g\varepsilon(h)$ for all $h,g \in H$  and hence $H\triangleright H^{op}$ is trivial.
\end{ex}

Now consider the double cross product $H = A\bowtie B$, $Z(A)$, $C_H(A)$ and $N_H(B)$ the center of $A$, the centralizer of $A$ in $H$ and the normal core of $B$ in $H$ respectively. Then $C_H(A)$ satisfies the following:

\begin{pr}
\label{centralizer}
Let $H = A\bowtie B$ be a double cross product such that $d(A,H) \leq 2$. Then 
\[C_H(A) = Z(A)\bowtie N_H(B)\]
as algebras
\end{pr}

\begin{proof}
Let $f\bowtie k \in C_H(A)$ and $a\bowtie 1_B \in A$. Then $(f\bowtie k)(a\bowtie 1_B) = (a\bowtie 1_B)(f\bowtie k)$.  On one hand we have
\[  (a\bowtie 1_B)(f\bowtie k) = af_1\bowtie (1_B\triangleleft f_2)k =af\bowtie k\]
Since depth two implies $A\triangleleft B$ is trivial. On the other hand 
\[(f\bowtie k)(a\bowtie 1_B) = f(k_1\triangleright a) \bowtie k_2\]
Now 
\[f(k_1\triangleright a) \bowtie k_2 = af \bowtie k\] 
if and only if $k\triangleright a = \varepsilon(k)a$ and $fa = af$ for all $a\in A$ if and only if $k \in N_H(B)$ and $f \in Z(A)$. 
\end{proof}

\begin{co}
\label{centralizergroup}
Let $kG$ be a finite group algebra and consider $H = D(kG)$ its Drinfel\vtick d double. Then
\[C_H((kG)^*) = Z((kG)^*)\bowtie Z(kG)\]
as algebras. 
\end{co}

\end{document}